\documentclass[12pt,reqno]{amsart}
\usepackage{amsmath,amssymb,amsthm,mathtools,wasysym,calc,verbatim,enumitem,tikz,pgfplots,url,hyperref,mathrsfs,cite,fullpage}

\newtheorem{theorem}{Theorem}[section]
\newtheorem{lemma}[theorem]{Lemma}
\newtheorem{corollary}[theorem]{Corollary}

\theoremstyle{definition}

\theoremstyle{remark}
\newtheorem{remark}[theorem]{Remark}

\newcommand{\one}{\mathbf{1}}

\renewcommand{\leq}{\leqslant}
\renewcommand{\geq}{\geqslant}

	\def\E{\mathbb{E}}

	\def\PP{\mathbb{P}}

\begin{document}

\title{A Short Note on the Average Maximal Number of Balls in a Bin}
\author{Marcus Michelen}
\address{Dept. of Mathematics, University of Pennsylvania, 
			209 South 33rd Street, Philadelphia, PA 19104. }
\email{marcusmi@sas.upenn.edu}
\maketitle

\begin{abstract}
We analyze the asymptotic behavior of the average maximal number of balls in a bin obtained by throwing uniformly at random $r$ balls without replacement into $n$ bins, $T$ times.  Writing the expected maximum as $\frac{r}{n}T+ C_{n,r}\sqrt{T} + o(\sqrt{T})$, a recent preprint of Behrouzi-Far and Zeilberger asks for an explicit expression for $C_{n,r}$ in terms of $n,r$ and $\pi$.  In this short note, we find an expression for $C_{n,r}$ in terms of $n, r$ and the expected maximum of $n$ independent standard Gaussians.  This provides asymptotics for large $n$ as well as closed forms for small $n$---e.g. $C_{4,2} = \frac{3}{2 \pi^{3/2}} \arccos(-1/3)$---and shows that computing a closed form for $C_{n,r}$ is precisely as hard as the difficult question of finding the expected maximum of $n$ independent standard Gaussians.   
\end{abstract}
\maketitle

\section{Introduction}
Suppose that you have $n$ bins, and in each round, you throw $r$ balls such that each ball lands in a
different bin, with each of the $\binom{n}{r}$
possibilities equally likely.  After $T$ rounds, set $U(n,r ; T)$ to be the maximum occupancy among the $n$ bins.  Set $A(n,r; T) = \E U(n,r; T) - \frac{r}{n} T$, and suppose that $A(n,r;T) =  C_{n,r}\sqrt{T} + o(\sqrt{T})$.  A recent preprint \cite{BFZ} of Behrouzi-Far and Zeilberger asks for an explicit expression for $C_{n,r}$ in terms of $n,r$ and $\pi$; they also calculate estimates for $C_{2,1},C_{3,1},C_{4,1},C_{4,2}$ using recurrence relations derived with computer aid.  As a motivation, Behrouzi-Far and Zeilberger \cite{BFZ} note that this problem arises in computer systems since load distribution across servers can be modeled with balls and bins.

Rather than utilizing exact computation in the vein of \cite{BFZ}, we use a multivariate central limit theorem to prove the following:  

\begin{theorem}\label{th:main}
	$$C_{n,r} := \lim_{T \to \infty}\frac{ A(n,r ; T)}{\sqrt{T}} = \sqrt{\frac{r(n-r)}{n(n-1)}} \E\left[ \max_{1 \leq j \leq n} Z_j \right]$$
	where $Z_j$ are i.i.d. standard Gaussians.  
\end{theorem}

The expected maximum of $n$ i.i.d. standard Gaussians appears to have no known closed form for general $n$ and in fact the known forms for small $n$ can be quite nasty; for instance, when $n = 5$, the expected value is $\frac{5}{2 \pi^{3/2}} \arccos(-23/27)$.  A short table of computed values is included in Section \ref{sec:numbers}.

From here, we extract asymptotics for $n \to \infty$, uniformly in $r$: \begin{corollary}
	As $n \to \infty$, we have 
	$$C_{n,r} \sim \sqrt{\frac{2r(n-r)\log(n)}{n^2}} $$
	uniformly in $r$.
\end{corollary}
\begin{proof}
	This follows from utilizing $\E[\max_{1 \leq j \leq n} Z_j  ] \sim \sqrt{ 2 \log(n) }$ (see, for instance, \cite[Exercise $3.2.3$]{durrett}).
\end{proof}

The exact form in Theorem \ref{th:main} also picks up a nice combinatorial property: \begin{corollary}
	For each $n$, the sequence $\{C_{n,r}\}_{r=1}^{n-1}$ is $\log$-concave. 
\end{corollary}
\begin{proof}
	Log-concavity follows from the inequality $$(r-1)(n-r+1)(r+1)(n-r-1) = (r^2 - 1)((n-r)^2 - 1)\leq r^2 (n-r)^2\,.$$	
\end{proof}

To prove Theorem \ref{th:main}, we use a multivariate central limit theorem to prove a limit theorem for $\frac{U(n,r;T) - \frac{r}{n} T}{\sqrt{T}}$ (Corollary \ref{cor:conv-in-dist}), show that we can exchange the limit and expectation (Lemma \ref{lem:form}), and then relate this expectation to the expected maximum of i.i.d. standard normals (Lemma \ref{lem:comp-Y}). 

\section{Proving Theorem \ref{th:main}}

Set $b(n,r;T)$ to be the random vector in $\{0,1,\ldots,T \}^n$ denoting the occupancies of the bins at time $T$.  The following representation for $b(n,r; T)$ is immediate:

\begin{lemma}\label{lem:b-rep}
	Fix $n,r$ and let $X$ be the random variable in $\{0,1\}^n$ chosen uniformly among vectors $v \in \{0,1\}^n$ with $\|v\|_{L^1} = r$.  Let $X_1,X_2,\ldots$ be i.i.d. copies of $X$.  Then $b(n,r;T) \stackrel{d}{=} \sum_{j = 1}^T X_j$.  Further, the random variable $X$ has covariance matrix $\Gamma$ given by $$\Gamma_{i,j} = \begin{cases}
	\frac{r(n-r)}{n^2} &\text{ for }i = j \\
	-\frac{r(n-r)}{n^2(n-1)} &\text{ for }i \neq j
	\end{cases}\,.$$
\end{lemma}
\begin{proof}
	The covariance matrix $\Gamma$ can be calculated easily: $$\Gamma_{j,j} = \frac{r}{n}\left(1 - \frac{r}{n}\right) = \frac{r(n-r)}{n^2}\,.$$
	
\noindent	For $\Gamma_{i,j}$ with $i \neq j$, we compute $$\Gamma_{i,j} = \frac{\binom{n-2}{r-2}}{\binom{n}{r}} - \frac{r^2}{n^2} = -\frac{r(n-r)}{n^2(n-1)}\,.$$
\end{proof}

From here, the multivariate central limit theorem shows convergence in distribution.  

\begin{corollary}\label{cor:conv-in-dist}
	$$\frac{U(n,r;T) - \frac{r}{n} T }{\sqrt{T}} \xrightarrow{d} \max\{Y_1,\ldots, Y_n \}$$
	
	where $(Y_1,\ldots,Y_n)$ is a mean-zero multivariate Gaussian with covariance matrix $\Gamma$, given in Lemma \ref{lem:b-rep}.  
\end{corollary}
\begin{proof}
	The multivariate central limit theorem \cite[Theorem $3.9.6$]{durrett} implies that $$\frac{b(n,r;T) - \E b(n,r;T)}{\sqrt{T}} \to \mathcal{N}(0,\Gamma)\,.$$
	The identity $\E b(n,r;T) = (\frac{r}{n},\ldots,\frac{r}{n})$ together with the continuous mapping theorem implies the Corollary.	
\end{proof}

To gain information about $A(n,r; T)$, we need to show that not only do we have convergence in distribution, but that we can switch the order of taking limits and expectation.

\begin{lemma}\label{lem:form}
	$$C_{n,r} := \lim_{T \to \infty}\frac{ A(n,r ; T)}{\sqrt{T}} = \E \max\{Y_1,\ldots,Y_n\}  $$
	where $(Y_1,\ldots,Y_n)$ are jointly Gaussian with mean $0$ and covariance matrix given by $\Gamma$ as defined in Lemma \ref{lem:b-rep}.
\end{lemma}
\begin{proof}
	Our strategy is to show uniform integrability of $\widehat{U}(T) := (U(n,r;T) - \frac{r}{n} T)/\sqrt{T}$; for $j \in \{1,2,\ldots,n\}$, let $b^{(j)}$ denote the number of balls in bin $j$.  Then by a union bound, we have \begin{equation}
	\PP\left[ \left|U(n,r;T) - \frac{r}{n}T \right| \geq \lambda \sqrt{T} \right] \leq n \PP\left[ \left|b^{(1)} - \frac{r}{n}T\right| \geq \lambda \sqrt{T} \right]\,.
	\end{equation}
	
\noindent	By Hoeffding's inequality (e.g. \cite[Theorem 7.2.1]{pm}), we bound $$\PP\left[ \left|b^{(1)} - \frac{r}{n}T\right| \geq \lambda \sqrt{T} \right] \leq 2 \exp\left(-2 \lambda^2  \right)\,.$$
	
\noindent	Thus, for each $T$ and $K > 0$ we have $$\E\left[ |\widehat{U}(T)|\cdot \one_{|\widehat{U}(T)| \geq K  } \right] \leq 2n \int_{K}^\infty e^{-2 \lambda^2}\,d\lambda\,.  $$
	
	This goes to zero uniformly in $T$ as $K \to \infty$, thereby showing that the family $\{\widehat{U}(T) \}_{T \geq 0}$ is uniformly integrable.  Since uniform integrability together with convergence in distribution implies convergence of means, Corollary \ref{cor:conv-in-dist} completes the proof.
\end{proof}

All that remains now is to relate $\E \max \{Y_1,\ldots,Y_n  \}$ to the right-hand-side of Theorem \ref{th:main}.

\begin{lemma}\label{lem:comp-Y}
	Let $(Y_1,\ldots,Y_n)$ be jointly Gaussian with mean $0$ and covariance matrix $\Gamma$.  Then $$\E[ \max \{ Y_1,\ldots,Y_n\}] = \sqrt{\frac{r(n-r)}{n(n-1)}} \E\left[\max_{1 \leq j \leq n}Z_j \right]$$
	where the variables $Z_j$ are i.i.d. standard Gaussians.
\end{lemma}
\begin{proof}
	Consider a multivariate Gaussian $(W_1,\ldots,W_n)$ with mean $0$ and covariance matrix given by $$\widetilde{\Gamma}_{i,j} = \begin{cases}
	\frac{n}{n-1} &\text{ for }i = j\\
	-\frac{n}{(n-1)^2} &\text{ for }i \neq j\,.
	\end{cases}$$
	
\noindent	Since $\Gamma = \frac{r(n-r)(n-1)}{n^3} \widetilde{\Gamma}$, we have \begin{equation}\label{eq:Y_n-W_n}
	(Y_1,\ldots,Y_n) \stackrel{d}{=} \sqrt{\frac{r(n-r)(n-1)}{n^3}} (W_1,\ldots,W_n)\,.\end{equation} The vector $(W_1,\ldots,W_n)$ can in fact be realized by setting $W_j = Z_j - \frac{\sum_{i \neq j} Z_i}{n-1}$ with $Z_i$ i.i.d. standard Gaussians.  This is because the two vectors are both mean-zero multivariate Gaussians and have the same covariance matrix.  Setting $S_n = \sum_{i = 1}^n Z_i$, we note $$W_j = -\frac{S_n}{n-1} + \frac{n}{n-1} Z_j$$
	
\noindent	thereby implying $$\max_{1 \leq j \leq n} \{W_j\} = -\frac{S_n}{n-1} + \left(\frac{n}{n-1} \right) \max_{1 \leq j \leq n} \{Z_j\}\,.$$
	
	\noindent Taking expectations and utilizing \eqref{eq:Y_n-W_n} completes the proof. 
\end{proof}

\begin{remark}
	The final piece of the proof of Lemma \ref{lem:comp-Y}---relating the expected maximum of the process $(Z_j - \frac{\sum_{i \neq j} Z_i}{n-1})_{j=1}^n$ to that of $(Z_j)_{j=1}^n$---is due to a Math Overflow answer of Iosef Pinelis \cite{pinelis}.
\end{remark}

\noindent \emph{Proof of Theorem \ref{th:main}:}  The theorem follows by combining Lemmas \ref{lem:form} and \ref{lem:comp-Y}. \qedhere

\section{Comparison with Numerics}  \label{sec:numbers}

Theorem \ref{th:main} proves an equality for $C_{n,r}$, although for large $n$, the expectation on the right-hand-side of Theorem \ref{th:main} appears to have no known closed form.  Calculating these values for small $n$ is tricky and tedious; we reproduce a few values of $\E[\max_{1 \leq j \leq n} Z_j]$ which can be computed precisely, as calculated in \cite{table-of-values}:\\
\begin{center}
\begin{tabular}{ r | l  }

	$n$ & $\E[\max_{1 \leq j \leq n} Z_j]$ \\ \hline
	2 & $\pi^{-1/2}$ \\ \hline
	3 & $(3/2) \pi^{-1/2}$ \\ \hline
	4 & $3 \pi^{-3/2} \arccos(-1/3)$ \\ \hline
	5 & $(5/2) \pi^{-3/2} \arccos(-23/27)$ \\
	
\end{tabular}
\end{center}
\medskip
We can then use these to obtain exact values for the values of $C_{n,r}$ predicted in \cite{BFZ}, and note that their predictions are quite close:
\medskip
\begin{center}
	\begin{tabular}{ c | c | c | c  }
		
		 & Exact Value & Numerical Approximation & Predicted Value from \cite{BFZ}  \\ \hline 
		$C_{2,1}$ & $\frac{1}{\sqrt{2\pi}}$  & $0.39894\ldots$ & $0.3989\ldots$\\ \hline
		$C_{3,1}$ & $\frac{\sqrt{3}}{2 \sqrt{\pi}}$  & $0.48860\ldots$ & $0.489\ldots$\\ \hline
		$C_{4,1}$ & $\frac{3}{2 \pi^{3/2}} \arccos(-1/3)$ & $0.51469\ldots$ & $0.516\ldots$ \\ \hline
		$C_{4,2}$ & $\frac{\sqrt{3}}{\pi^{3/2}}\arccos(-1/3)$ & $0.59431\ldots$ & $0.59430\ldots$
		
	\end{tabular}
\end{center}

\bibliographystyle{abbrv}
\bibliography{Bib}

\end{document}